\documentclass[11pt, onecolumn]{article}

\usepackage{constants}

\usepackage{stmaryrd}
\usepackage{tikz}
\usetikzlibrary{shapes.misc}
\tikzset{cross/.style={cross out, draw=black, minimum size=2*(#1-\pgflinewidth), inner sep=0pt, outer sep=0pt},
cross/.default={1pt}}

\usepackage{pstricks}
\usepackage[utf8]{inputenc}
\usepackage[T1]{fontenc}
\usepackage{lmodern}
\usepackage{xr}
\usepackage{bold-extra}
\usepackage{dsfont}

\usepackage{amsmath}
\usepackage{amssymb}
\usepackage{mathrsfs}
\usepackage{mathtools}

\usepackage{amsfonts}
\usepackage{amsthm}

\usepackage{enumerate}
\usepackage{multirow}
\usepackage{graphicx}
\usepackage{hyperref}
\usepackage{xcolor}
\hypersetup{
    colorlinks,
    linkcolor={black},
    citecolor={black},
    urlcolor={black}
}

\usepackage[top=2.5cm,bottom=2.5cm,left=2cm,right=2cm,marginparwidth=2.5cm]{geometry}
\reversemarginpar

\makeatletter
\renewcommand{\paragraph}{%
  \@startsection{paragraph}{4}%
  {\z@}{1.5ex \@plus 1ex \@minus .2ex}{-1em}%
  {\normalfont\normalsize\bfseries}%
}
\makeatother

{\theoremstyle{plain}
\newtheorem{Proposition}{\textbf{Proposition}}[section]

\newtheorem{Lemma}[Proposition]{\textbf{Lemma}}

 }

{\theoremstyle{plain}
\newtheorem{Theorem}[Proposition]{Theorem}}

{\theoremstyle{definition}
\newtheorem{Definition}[Proposition]{Definition}}

{\theoremstyle{remark}

\newtheorem{Remark}[Proposition]{Remark}

}

\numberwithin{equation}{section}

\newcommand{\eps}{\varepsilon}

\newcommand{\bbP}{\mathbb{P}}

\newcommand{\bbX}{\mathbb{X}}
\newcommand{\bbY}{\mathbb{Y}}

\newcommand{\N}{\mathbb{N}}
\newcommand{\Z}{\mathbb{Z}}
\newcommand{\R}{\mathbb{R}}

\renewcommand{\S}{\mathbb S} 

\newcommand{\cB}{\mathcal{B}}

\newcommand{\cG}{\mathcal{G}}

\newcommand{\cX}{\mathcal{X}}

\newcommand{\rL}{\mathrm{L}}

\let\limsup\relax

\DeclareMathOperator* \limsup {\overline{lim}}

\DeclareMathOperator*{\argmin}{arg\,min}

\newcommand{\declim}[1]{\lim_{#1}\!\!\downarrow\!}

\let\originalleft\left
\let\originalright\right
\renewcommand{\left}{\mathopen{}\mathclose\bgroup\originalleft}
\renewcommand{\right}{\aftergroup\egroup\originalright}

\newcommand{\p}[1]{\left( #1 \right)}
\newcommand{\acc}[1]{\left\{ #1 \right\}}
\newcommand{\cro}[1]{\left[ #1 \right]}

\newcommand{\set}[2]{\acc{#1 \;\middle\vert\; #2 } }

\newcommand{\ind}[1]{\mathds{1}_{#1}}

\newcommand{\dpe}{\coloneqq}

\newcommand{\eol}{\nonumber\\}

\newcommand{\module}[1]{\left\lvert #1 \right\rvert}
\newcommand{\norme}[2][]{\left\| #2 \right\|_{#1}}

\newcommand{\ball}[2][]{\mathrm{B}_{#1}\p{#2}}
\newcommand{\clball}[2][]{\overline{\mathrm{B}}_{#1}\p{#2}}

\newcommand{\intervalle}[4]{#1#2\,,#3#4}

\newcommand{\intervalleff}[2]{\intervalle{\left[}{#1}{#2}{\right]}}
\newcommand{\intervalleof}[2]{\intervalle{\left]}{#1}{#2}{\right]}}
\newcommand{\intervallefo}[2]{\intervalle{\left[}{#1}{#2}{\right[}}
\newcommand{\intervalleoo}[2]{\intervalle{\left]}{#1}{#2}{\right[}}

\newcommand{\intint}[2]{\left\llbracket#1\,,#2\right\rrbracket}

\renewcommand{\d}{\mathrm{d}}

\newcommand{\base}[1]{\mathrm e_{#1}}

\newcommand{\Leb}{\operatorname{Leb}}

\newcommand{\Dirac}[1]{\delta_{#1}}

\newcommand{\E}[2][]{\mathbb{E}_{#1} \left[ #2\right]}
\newcommand{\Econd}[3][]{\E[#1]{ #2 \middle| #3  }}

\newcommand{\Pb}[2][]{\mathbb{P}_{#1}\left( #2\right)}

\newcommand{\pro}{\mathcal{N}}
\newcommand{\projpro}{\mathcal{N^*}}
\newcommand{\ProSpace}[1][\R^d \times \intervalleoo0\infty]{\mathbf{N}\p{#1}}

\newcommand{\AUXAnimal}{\mathcal{A}}
\newcommand{\AUXPath}{\mathcal{P}}

\newcommand{\AUXGeneric}{\mathcal{G}}
\newcommand{\AUXMassAnimal}{\mathrm{A}}
\newcommand{\AUXMassPath}{\mathrm{P}}

\newcommand{\AUXMassGeneric}{\mathrm{G}}
\newcommand{\AUXLimAnimal}{\mathbf{A}}
\newcommand{\AUXLimPath}{\mathbf{P}}

\newcommand{\AUXLimGeneric}{\mathbf{G}}
\newcommand{\AUXBiancre}[3]{\p{#1 \leftrightarrow #2, #3}}

\newcommand{\AUXFree}[1]{#1}

\newcommand{\Mass}[1]{\mathbf{m}\p{#1}}

\newcommand{\SetADF}[3]{\AUXFree{\AUXAnimal}\AUXBiancre{#1}{#2}{#3} }

\newcommand{\SetAUF}[1]{\AUXFree{\AUXAnimal}\p{#1} }

\newcommand{\SetPDF}[3]{\AUXFree{\AUXPath}\AUXBiancre{#1}{#2}{#3} }

\newcommand{\SetPUF}[1]{\AUXFree{\AUXPath}\p{#1} }

\newcommand{\SetGDF}[3]{\AUXFree{\AUXGeneric}\AUXBiancre{#1}{#2}{#3} }

\newcommand{\SetGUF}[1]{\AUXFree{\AUXGeneric}\p{#1} }

\newcommand{\SetADFalt}[3]{\AUXFree{\AUXAnimal}^{*}\AUXBiancre{#1}{#2}{#3} }
\newcommand{\SetAUFalt}[1]{\AUXFree{\AUXAnimal}^{*}\p{#1} }

\newcommand{\MassADF}[3]{\AUXFree{\AUXMassAnimal}\AUXBiancre{#1}{#2}{#3} }

\newcommand{\MassAUF}[1]{\AUXFree{\AUXMassAnimal}\p{#1} }

\newcommand{\MassPDF}[3]{\AUXFree{\AUXMassPath}\AUXBiancre{#1}{#2}{#3} }

\newcommand{\MassPUF}[1]{\AUXFree{\AUXMassPath}\p{#1} }

\newcommand{\MassGDF}[3]{\AUXFree{\AUXMassGeneric}\AUXBiancre{#1}{#2}{#3} }

\newcommand{\MassGUF}[1]{\AUXFree{\AUXMassGeneric}\p{#1} }

\newcommand{\MassADFpen}[4]{\AUXFree{\AUXMassAnimal}^{(#4)}\AUXBiancre{#1}{#2}{#3} }
\newcommand{\MassAUFpen}[2]{\AUXFree{\AUXMassAnimal}^{(#2)}\p{#1} }

\newcommand{\LimMassG}{\AUXLimGeneric }

\newcommand{\LimMassA}{\AUXLimAnimal }

\newcommand{\LimMassP}{\AUXLimPath}

\newcommand{\GeneralUB}{\mathbf{M}}

\newcommand{\Good}{\cG}
\newcommand{\Bad}{\cB}

\title{Law of large numbers for greedy animals and paths in a Poissonian environment}
\author{Julien \textsc{Verges}}
\date{\today}

\begin{document}

\maketitle
\abstract{We study two continuous and isotropic analogues of the model of greedy lattice animals introduced by Cox, Gandolfi, Griffin and Kesten~\cite{Cox93, Gan94} in 1993-94. In our framework, animals collect masses scattered on a Poisson point process on $\mathbb R^d$, and are allowed to have vertices outside the process or not, depending on the model. The author proved in~\cite{Ver24LLN_Greedy} a more general setting that for all $u$ in the Euclidean open unit ball, the maximal mass of animals with length $\ell$, containing $0$ and $\ell u$ satisfies a law of large numbers. We prove some additional properties in the Poissonian case, including an extension of the functional law of large numbers to the closed unit ball, and study strict monotonicity of the limit function along a radius. Moreover, we prove that a third, penalized model is a suitable interpolation between the former two.}

\section{Introduction}
\label{sec : intro}

\subsection{Context}

In 1993-94, Cox, Gandolfi, Griffin and Kesten~\cite{Cox93, Gan94} introduced the models of \emph{greedy lattice animals} and \emph{greedy lattice paths} as such: given an integer $d\ge2$ and a family of masses indexed by the vertices of the standard lattice $\Z^d$, the mass of a lattice animal (i.e.\ a finite connected subset of $\Z^d$) or path is defined as the sum of the masses of its vertices. They proved that under a moment condition, the maximal mass of lattice animals (resp.\ paths) containing (resp.\ starting from) the origin with cardinal $n$ satisfies a law of large numbers. This moment condition was improved in 2002 by Martin~\cite{Mar02}.

The notion of greedy \emph{continuous} paths has been introduced in 2008 by Gouéré and Marchand~\cite{Gou08} as a tool for the study of a continuous first-passage percolation model. Masses are scattered on a Poisson point process on $\R^d$ instead of $\Z^d$, and the lattice paths are replaced by polygonal lines. They proved that the maximal mass of paths with fixed length $\ell$ grows at most linearly in $\ell$. Gouéré and Théret~\cite{Gou17} also used this result, for a further study of the same model.

The author extended in~\cite{Ver24LLN_Greedy} the results of Cox, Gandolfi, Griffin, Kesten and Martin in two ways:
\begin{enumerate}
 	\item The family of masses may be any ergodic marked point process on $\R^d \times \intervalleoo0\infty$, provided integrability. Marked Poisson point processes belong to this setting.
 	\item The author proved a law of large numbers for the maximal mass of animals (resp.\ paths) of length $n$, containing (resp.\ with endpoints) $0$ and $n u$, with $u\in \ball{0,1}$, uniformly with respect to $u$ on a compact subset of $\ball{0,1}$.
\end{enumerate}
Three extensions of greedy lattice animals are proposed in this context. The first one defines continous animals as finite graphs with vertices in $\R^d$, the second one is a restriction to animals whose vertices are atoms of the point process, and the third one is an interpolation of the preceding two, where an animal may have vertices outside the point process, but each of them brings a penalization $-q$ in the definition of the mass. Section~\ref{subsec : intro/framework} gives the formal definitions. Theorems~\ref{thm : intro/known} and~\ref{thm : intro/known_pen} summarize the results proven in~\cite{Ver24LLN_Greedy} in the case of Poisson processes.

This article aims to establish extra properties in the Poissonian case, namely:
\begin{enumerate}
 	\item The limit functions have continuous extensionz to the closed ball unit ball, and vanish at the boundary. Moreover, they are strictly decreasing along a radius, on an explicit interval (see Theorem~\ref{thm : intro/main/limit_functions}).
 	\item The functional law of large numbers given in~\cite{Ver24LLN_Greedy} may be extended to the closed unit ball (see Theorem~\ref{thm : intro/main/cvg}).
 	\item The penalized model actually converges to the restricted one when the penalization parameter $q$ tends to $\infty$ (see~Theorem~\ref{thm : intro/main/q}).
\end{enumerate}

\subsection{Framework}
\label{subsec : intro/framework}

Let $d\ge 2$ be a integer and $\norme\cdot$ be the Euclidean norm on $\R^d$. For all $x\in \R^d$ and $r>0$, let $\ball{x,r}$ and $\clball{x,r}$ respectively denote the Euclidean open and closed balls of center $x$ and radius $r$. For all subsets $A,B\subseteq \R^d$, we define
\begin{equation}
 	\d(A,B) \dpe \inf \set{\norme{x-y} }{x\in A,\quad y\in B}.
\end{equation}
Let $\Leb$ denote the Lebesgue measure on $\R^d$. 

\paragraph{Point processes.} Let $\ProSpace$ denote the space of measures on $\R^d\times \intervalleoo0\infty$ which take integer values on compact subsets, endowed with the $\sigma$-algebra generated by the maps $\eta \mapsto \eta(A)$, for all Borel subsets $A\subseteq \R^d\times \intervalleoo0\infty$. We call \emph{marked point process} (on $\R^d\times \intervalleoo0\infty$) a random variable $\pro$ with values in $\ProSpace$.

Let $\nu$ be a nontrivial locally finite measure on $\intervalleoo0\infty$ that satisfies
\begin{equation}
	\label{eqn : intro/context/greedy_condition}
	\int_0^\infty \nu\p{\intervallefo t\infty}^{1/d}\d t< \infty
\end{equation}
and $\pro$ be a Poisson point process $\R^d \times \intervalleoo0\infty$ with intensity $\Leb \otimes \nu$, i.e.\ a point process on $\R^d\times \intervalleoo0\infty$ such that
\begin{enumerate}[(i)]
	\item For all disjoint Borel subsets $A_1,\dots, A_k \subseteq \R^d \times \intervalleoo0\infty$, $\pro(A_i)$, for $1\le i \le k$ are independent.
	\item For all Borel subsets $A\subseteq \R^d\times \intervalleoo0\infty$, $\pro(A)$ follows a Poisson distribution with parameter $\Leb\otimes \nu (A)$.
\end{enumerate}
Such a process always exists (see e.g Proposition~2.2.4  in \cite{Bac20}), and its distribution is determined by $\nu$. Moreover, it is simple (see e.g.\ Proposition~2.1.9 in \cite{Bac20}) and (see e.g.\ Lemma 1.6.8  in \cite{Bac20}) it may almost surely be written as the sum
\begin{equation}
	\label{eqn : intro/main/decomposition_pro}
	\pro = \sum_{n=1}^{\infty} \Dirac{(z_n, \Mass{z_n})},
\end{equation} 
where the $z_n$ and the $\Mass{z_n}$ for $n\ge 1$ are random variables with values in $\R^d$ and $\intervalleoo0\infty$ respectively.

\begin{Definition}
	\label{def : intro/framework/mass_path_animal}
	For every subset $A\subseteq \R^d$, the \emph{mass} of $A$ is defined as
	\begin{equation}
		\label{eqn : intro/framework/mass_path_animal}
		\Mass{A} \dpe \int_{A \times \intervalleoo0\infty} t \pro(\d x, \d t) = \sum_{n=1}^N \Mass{z_n}\ind{z_n \in A},
	\end{equation}
	with the notations of~\eqref{eqn : intro/main/decomposition_pro}.
\end{Definition}

\paragraph{Continuous paths and animals.}

\begin{Definition}
	\label{def : intro/framework/path_animal}
	Following Gouéré and Marchand (2008) \cite{Gou08}, we call (continuous) \emph{path} a finite sequence of points of $\R^d$. We define the \emph{length} of a path $\gamma = (x_0,\dots, x_r)$ as
	\begin{equation}
	\label{eqn : intro/framework/length_path}
		\norme\gamma \dpe \sum_{i=0}^{r-1}\norme{x_i - x_{i+1}}.
	\end{equation}
	A (continuous) \emph{animal} is a finite connected graph $\xi =(V,E)$, whose vertex set $V$ is included in $\R^d$. We define its \emph{length} as
	\begin{equation}
	\label{eqn : intro/framework/length_animal}
		\norme\xi \dpe \sum_{\acc{x,y}\in E} \norme{x-y}.
	\end{equation}
\end{Definition}
When there is no ambiguity, we will identify a path or an animal with its vertex set, e.g.\ for any animal $\xi =(V,E)$, $\Mass{\xi}= \Mass{V}$ and for any path $\gamma=(x_0,\dots, x_r)$, $\Mass{\gamma}=\Mass{\acc{x_0,\dots, x_r}}$. Similarly, $\d (0, \xi)$ denotes the distance between $0$ and the \emph{vertex set} of $\xi$. We are interested in the following families of animals and paths. For all $x,y\in \R^d$ and $\ell\ge0$, we define:
\begin{itemize}
	\item  $\SetPUF{\ell}$ as the set of paths of length at most $\ell$, starting at $0$.
	\item  $\SetPDF{x}{y}{\ell}$ as the set of paths of length at most $\ell$, starting at $x$ and ending at $y$.
\end{itemize}
Likewise, we define:
\begin{itemize}
	\item  $\SetAUF{\ell}$ as the set of animals of length at most $\ell$, containing $0$.
	\item  $\SetADF{x}{y}{\ell}$ as the set of animals of length at most $\ell$, containing $x$ and $y$.
\end{itemize}

\paragraph{The processes.}
For any set of paths or animals denoted by a calligraphic font letter, we use the same letter in roman typestyle to denote the supremum of the mass of animals or paths in this set. For example, for all $\ell\ge 0$,
\begin{equation}
	\label{eqn : intro/framework/def_MassPUF}
	\MassPUF{\ell} \dpe \sup_{\gamma \in \SetPUF{\ell} } \Mass{\gamma}.
\end{equation}
We also use this convention for a generic $\AUXGeneric \in \acc{\AUXPath, \AUXAnimal}$: for all $x,y\in \R^d$ and $\ell\ge 0$,
\begin{equation}
	\label{eqn : intro/framework/def_MassGUF_MassGDF}
	\MassGUF{\ell} \dpe \sup_{\gamma \in \SetGUF{\ell} } \Mass{\gamma}%
	\text{ and }
	\MassGDF{x}{y}{\ell} \dpe \sup_{\gamma \in \SetGDF{x}{y}{\ell}} \Mass{\gamma}.
\end{equation}
Another natural definition of the continous animals model consists in restricting the supremum to animals which vertex sets are included in the set 
\begin{equation}
	\projpro \dpe \set{x \in \R^d}{\Mass{x}>0}.
\end{equation}
More precisely, for all $x,y\in\R^d$ and $\ell\ge0$, we define:
\begin{itemize}
	\item $\SetAUFalt{\ell}$ as the set of animals $\xi$ such that $\norme\xi + \d(0,\xi) \le \ell$, or $\xi$ is empty,
	\item $\SetADFalt{x}{y}{\ell}$ as the set of animals $\xi$ such that $\norme\xi + \d(x,\xi) + \d(y,\xi)\le \ell$, or $\xi$ is empty,
\end{itemize}
and the corresponding variables 
\begin{equation}
	\label{eqn : intro/framework/animaux_contraints}
	\MassAUFpen\ell{\infty} \dpe \sup_{\substack{ \xi \in \SetAUFalt{\ell} \\ \xi \subseteq \projpro}  } \Mass\xi%
	\text{ and }
 	\MassADFpen xy\ell{\infty} \dpe \sup_{\substack{ \xi \in \SetADFalt xy{\ell} \\ \xi \subseteq \projpro}  } \Mass\xi.
\end{equation}
It is pointless to introduce similar processes for paths, since by triangle inequality, skipping vertices outside $\projpro$ along a path produces a path with the same mass and smaller length. We also introduce a third model which is an interpolation of the preceding two. For all $q\in\intervalleff0\infty$, $x,y\in\R^d$ and $\ell \ge 0$, we define
\begin{align}
	\label{eqn : intro/framework/animaux_pen1}
	\MassAUFpen{\ell}{q} &\dpe \sup_{\xi \in \SetAUFalt{\ell}} \cro{ \Mass{\xi} - q\#\p{\xi \cap \projpro^\mathrm{c}} }\\%
	\text{and }
	\label{eqn : intro/framework/animaux_pen2}
	\MassADFpen{x}{y}{\ell}{q} &\dpe \sup_{\xi \in \SetADFalt{x}{y}{\ell}} \cro{ \Mass{\xi} - q\#\p{\xi \cap \projpro^\mathrm{c}} }, 
\end{align}
where $\# A$ denotes the cardinal of a set $A$ and $\xi$ is identified to its vertex set in the expression $\xi \cap \projpro^\mathrm{c}$. In other words $\MassAUFpen{\ell}{q}$ and $\MassADFpen{x}{y}{\ell}{q}$ are analogues of $\MassAUF{\ell}$ and $\MassADF{x}{y}{\ell}$, with a penalization $-q$ for every vertex of $\xi$ not belonging to the point process. By adding the vertex $0$ and one edge, one shows that any animal in $\SetAUFalt{\ell}$ is included in an animal in $\SetAUF{\ell}$, thus $\MassAUFpen{\ell}{0} \le \MassAUF{\ell}$. The inclusion $\SetAUF{\ell} \subseteq \SetAUFalt{\ell}$ gives the converse inequality, hence
\begin{align*}
	\MassAUFpen{\ell}{0} &= \MassAUF{\ell}.
	\intertext{Likewise,}
	\MassADFpen{x}{y}{\ell}{0} &= \MassADF{x}{y}{\ell}.
\end{align*} 
Besides, \eqref{eqn : intro/framework/animaux_contraints} is compatible with \eqref{eqn : intro/framework/animaux_pen1} and \eqref{eqn : intro/framework/animaux_pen2}.

Note that for all $\ell>0$ and $q\in\intervalleff0\infty$,
\begin{equation}
	\label{eqn : intro/framework/easy_inequality}
	\MassPUF\ell \le \MassAUFpen\ell q \le \MassAUF\ell \le \MassPUF{2\ell},
\end{equation}
where we have used in the last inequality the fact that any animal may be covered by the path obtained by a depth-first search.

\paragraph{Law of large numbers.} Let $\cX$ denote the subset of $\clball{0,1}^2\times \intervalleof01$ consisting of triplets $(x,y,\ell)$ such that $\norme{x-y}< \ell$, and $x$ and $y$ are colinear. Corollary~1.8 in~\cite{Ver24LLN_Greedy} and the invariance of the model by rotations imply Theorems~\ref{thm : intro/known} and~\ref{thm : intro/known_pen}.

\begin{Theorem}
	\label{thm : intro/known}
	Assume that $\nu$ satisfies~\eqref{eqn : intro/context/greedy_condition}. Let $\AUXGeneric \in \acc{\AUXPath, \AUXAnimal}$. Then
	there exists a deterministic, nonincreasing, concave function $\LimMassG : \intervallefo01 \rightarrow \intervallefo0\infty$, such that for all compact subsets $K\subseteq \cX$,
	\begin{equation}
		\label{eqn : intro/known/cvg}
		\sup\set{ \module{ \frac{\MassGDF{Lx}{Ly}{L\ell} }{L} - \ell\LimMassG\p{\frac{ \norme{x-y} }{\ell}}  } }%
		{(x,y,\ell) \in K}%
		\xrightarrow[L\to \infty]{\text{a.s.\ and }\rL^1} 0.
	\end{equation}
	Moreover,
	\begin{equation}
		\label{eqn : intro/known/cvg_undirected}
		\frac{\MassGUF L}{L} \xrightarrow[L\to \infty]{\text{a.s.\ and }\rL^1} \LimMassG(0),
	\end{equation}
	and 
	\begin{equation}
		\label{eqn : intro/known/General_bound}
		\GeneralUB \dpe \sup_{\ell >0} \frac{\E{\MassAUF\ell} }{\ell} \le \Cr{Dirac_bound} \int_0^\infty \nu\p{\intervallefo t\infty}^{1/d}\d t,
	\end{equation}
	where $\Cl{Dirac_bound}>0$ is a constant that only depends on $d$.
\end{Theorem}
\begin{Theorem}
	\label{thm : intro/known_pen}
	Assume that $\nu$ satisfies~\eqref{eqn : intro/context/greedy_condition}. Fix $q\in \intervalleff0\infty$. Then
	there exists a deterministic, nonincreasing, concave function $\LimMassA^{(q)} : \intervallefo01 \rightarrow \intervallefo0\infty$, such that for all compact subsets $K\subseteq \cX$,
	\begin{equation}
		\label{eqn : intro/known_pen/cvg}
		\sup\set{ \module{ \frac{\MassADFpen{Lx}{Ly}{L\ell}q }{L} - \ell\LimMassA^{(q)}\p{\frac{ \norme{x-y} }{\ell}}  } }%
		{(x,y,\ell) \in K}%
		\xrightarrow[L\to \infty]{\text{a.s.\ and }\rL^1} 0.
	\end{equation}
	Moreover,
	\begin{equation}
		\label{eqn : intro/known_pen/cvg_undirected}
		\frac{\MassAUFpen Lq}{L} \xrightarrow[L\to \infty]{\text{a.s.\ and }\rL^1} \LimMassA^{(q)}(0).
	\end{equation}
\end{Theorem}

\subsection{Main results}

Our first result, Theorem~\ref{thm : intro/main/limit_functions}, yields extra information on the limit functions $\LimMassP$ and $\LimMassA^{(q)}$.

\begin{Theorem}
	\label{thm : intro/main/limit_functions}
	Assume that $\nu$ satisfies~\eqref{eqn : intro/context/greedy_condition}. Then
	\begin{enumerate}[(i)]
		\item Setting $\LimMassP(1) = 0$ and $\LimMassA^{(q)}(1) = 0$, for $q\in\intervalleff0\infty$, gives a continuous extension of these functions on $\intervalleff01$.
		\item The functions $\LimMassP$ and $\LimMassA^{(q)}$, for $q\in\intervalleff0\infty$ are strictly decreasing on $\intervalleff{\frac{1}{\sqrt d}}{1}$.
	\end{enumerate}
\end{Theorem}
We believe that $\LimMassP$ and $\LimMassA^{(q)}$ are actually strictly decreasing on $\intervalleff01$, since smaller values of $\beta$ essentially means lesser constraints on the candidate animals or paths. Our proof, however, fails on $\intervalleff0{\frac{1}{\sqrt d}}$ because it is based on a deterministic constraint on paths in $\SetPDF{0}{\ell \beta \base 1}{\ell}$ that only holds for $\beta \ge \frac{1}{\sqrt d}$ (see Lemma~\ref{lem : poisson/stretching} and Remark~\ref{rmk : poisson/stretching}).

Theorem~\ref{thm : intro/main/cvg} improves Theorems~\ref{thm : intro/known} and~\ref{thm : intro/known_pen} by extending the set of admissible triplets $(x,y,\ell)$ in order to include the boundary cases where $\ell = \norme{x-y}$. 
\begin{Theorem}
	\label{thm : intro/main/cvg}
	Assume that $\nu$ satisfies~\eqref{eqn : intro/context/greedy_condition}. Then
	\begin{enumerate}[(i)]
		\item 	\begin{equation}
					\label{eqn : intro/main/cvg}
					\sup\set{ \module{\frac{\MassGDF{Lx}{Ly}{L\ell} }{L} - \ell \LimMassG\p{\frac{ \norme{x-y} }{\ell}} } }%
					{\begin{array}{c}
						x,y\in \clball{0,1} \text{ colinear}\\ \norme{x-y} \le \ell \le 1	\end{array} }%
					\xrightarrow[L \to \infty]{\text{a.s.\ and }\rL^1} 0.
				\end{equation}
		\item For all $q\in \intervalleff0\infty$,
				\begin{equation}
					\label{eqn : intro/main/cvg_pen}
					\sup\set{ \module{\frac{\MassADFpen{Lx}{Ly}{L\ell}q }{L} - \ell \LimMassA^{(q)}\p{\frac{ \norme{x-y} }{\ell}} } }%
					{\begin{array}{c}
						x,y\in \clball{0,1} \text{ colinear}\\ \norme{x-y} \le \ell \le 1	\end{array} }%
					\xrightarrow[L \to \infty]{\text{a.s.\ and }\rL^1} 0.
				\end{equation}
	\end{enumerate}
\end{Theorem}

Finally, Theorem~\ref{thm : intro/main/q} states that the penalized model constitutes a good interpolation between the two natural models of continuous greedy animals (namely with or without authorizing animals to have vertices outside the point process), in the sense that making the penalization parameter $q$ vary from $0$ to $\infty$ continously transforms the former into the latter.
\begin{Theorem}
	\label{thm : intro/main/q}
	Assume that $\nu$ satisfies~\eqref{eqn : intro/context/greedy_condition}. Then under the topology of the uniform convergence on $\intervalleff01$, $q\mapsto \LimMassA^{(q)}$ is continuous on $\intervalleff0\infty$.
\end{Theorem}

\subsection{Outline of the paper}
In Section~\ref{sec : poisson/Preliminary} we prove Lemmas~\ref{lem : poisson/scaling} and~\ref{lem : poisson/stretching}, which respectively state that:
\begin{enumerate}
	\item In distribution, applying an homothety to $\pro$ or multiplying $\nu$ by a well-chosen constant is equivalent.
	\item For all $\beta\in \intervallefo{\frac{1}{\sqrt d}}1$, $\LimMassP(\beta)$ is upper bounded by a quantity which is strictly lesser than $\LimMassP(0)$ and vanishes as $\beta \to1$. 
\end{enumerate}
Our arguments are based on the so-called mapping theorem (see e.g.\ Theorem~5.1 in~\cite{LastPenrose}), which states that the image of a Poisson process by any measurable map is a Poisson process. These lemmas do not transpose well the case of general point processes.

Section~\ref{subsec : poisson/limit_functions} is devoted to the proof of Theorem~\ref{thm : intro/main/limit_functions}. We use the already known concavity and monotonicity of the limit functions, together with Lemma~\ref{lem : poisson/stretching}.

Section~\ref{subsec : poisson/cvg} contains the proof of Theorem~\ref{thm : intro/main/cvg}. Given Theorems~\ref{thm : intro/known_pen} and~\ref{thm : intro/main/limit_functions}, it is sufficient to upper bound the mass of an animal or path in $\MassGDF{Lx}{Ly}{L\ell}$, when $\frac{\norme{x-y}}{\ell} \simeq 1$.

In Section~\ref{subsec : poisson/q} we prove Theorem~\ref{thm : intro/main/q}. For the continuity at $q=\infty$, we use a sprinkling argument: given a animal $\xi(\ell)$ which realizes the maximum in the definition of $\MassADFpen{0}{\ell \beta \base 1}{\ell}{q}$, we replace each vertex of $\xi(\ell) \cap (\pro^*)^\mathrm{c}$ by an atom of a new Poisson point process, independent of the first one and having low intensity. Another Poisson-specific result called the superposition principle (see e.g.\ Theorem~3.3 in \cite{LastPenrose}) and Lemma~\ref{lem : poisson/scaling} provides a link between the limit functions in the original environment and in the enriched one. For the continuity on $\intervallefo0\infty$, we show that except in the straightforward case where $\nu\p{\intervalleoo0\infty} = \infty$, the number of points of $\pro^*$ belonging to an animal grows sublinearly with respect to its length, thus giving Lipschitz continuity of the limit function with respect to $q$.
\subsection{Notations}
In contexts where more than one point process is considered, we will indicate the dependence on the point process by square brackets, e.g.\ for any point process $\pro'$ on $\R^d \times \intervallefo0\infty$ and any subset $A\subseteq \R^d$,
\begin{equation}
	\label{eqn : intro/framework/mass_subset_other_process}
	\Mass{A}\cro{\pro'} \dpe \int_{A \times \intervallefo0\infty} t \pro'(\d x, \d t).
\end{equation}
For objects that only depend on the distribution $\nu$, we will indicate this dependence by $\cro{\nu}$. We denote by $\bbP_\nu$ the distribution of $\pro$.

We denote by $\p{\base i}_{1\le i \le d}$ the canonical basis of $\R^d$.


\section{Proof of the results}
\label{sec : poisson}
We fix a measure $\nu$ on $\intervalleoo0\infty$ such that the moment assumption~\eqref{eqn : intro/context/greedy_condition} is satisfied and a Poisson point process $\pro$ on $\R^d\times \intervalleoo0\infty$ with intensity $\Leb\otimes \nu$.

\subsection{Two preliminary lemmas}
\label{sec : poisson/Preliminary}
We recall the so-called mapping theorem (see e.g.\ Theorem~5.1 in \cite{LastPenrose}). 
\begin{Theorem}
	\label{thm : poisson/mapping}
	Let $\bbX$ and $\bbY$ be measurable spaces, $f : \bbX \rightarrow \bbY$ a measurable function and $\eta$ a Poisson point process on $\bbX$ with intensity $\mu$. Then the image $\eta \circ f^{-1}$ of $\eta$ by $f$ is a Poisson point process on $\bbY$, with intensity $\mu\circ f^{-1}$.
\end{Theorem}
It is the crucial argument in the proofs of Lemmas~\ref{lem : poisson/scaling} and~\ref{lem : poisson/stretching}.
\begin{Lemma}[The scaling lemma]
	\label{lem : poisson/scaling}
	Let $\lambda>0$ and $q\in\intervalleff0\infty$. The distribution of $\p{ \MassADFpen{\lambda x}{\lambda y}{\lambda \ell}q }_{ \substack{x,y\in \R^d\\ \ell >0} }$ under $\bbP_{\nu}$ is equal to the distribution of $\p{\MassADFpen{x}{y}{\ell}q}_{ \substack{x,y\in \R^d\\ \ell >0} }$ under $\bbP_{\lambda^d\nu}$.
\end{Lemma}
\begin{proof}
	Recall the decompostion~\eqref{eqn : intro/main/decomposition_pro} of $\pro$. By Theorem~\ref{thm : poisson/mapping}, the point process 
	\begin{equation}
		\pro' \dpe \sum_{n=1}^\infty \Dirac{(\lambda^{-1}z_n, \Mass{z_n})}
	\end{equation}
	is a Poisson point process with intensity $\Leb\otimes \lambda^d \nu$. Let $x,y\in \R^d$ and $\ell>0$. For all animals $\xi$,
	\begin{align*}
		\xi \in \SetADFalt{\lambda x}{\lambda y}{\lambda \ell} &\iff  \lambda^{-1}\xi \in \SetADFalt{x}{y}{\ell} \\
		\text{and } \Mass{\xi}[\pro] -q\#\p{\xi \cap \projpro^\mathrm c}  &= \Mass{\lambda^{-1}\xi}[\pro'] -q\#\p{(\lambda^{-1}\xi) \cap {\pro'^*}^\mathrm c} ,
		\intertext{thus}
		\MassADFpen{\lambda x}{\lambda y}{\lambda \ell}{q}[\pro] &= \MassADFpen{x}{y}{\ell}{q}[\pro'].
	\end{align*}
\end{proof}
\begin{Lemma}[The stretching lemma]
	\label{lem : poisson/stretching}
	Let $\frac{1}{\sqrt d} \le \beta < 1$ and
	\begin{equation}
		\label{eqn : poisson/stretching_def_g}
		g(\beta) \dpe \sqrt d \beta^{1/d}\p{\frac{1-\beta^2}{d-1}}^{\frac{d-1}{2d}}.
	\end{equation}
	Then 
	\begin{align}
		\label{eqn : poisson/stretching_path}
		\LimMassP(\beta) &\le g(\beta) \LimMassP(0)%
		\intertext{and for all $q\in\intervalleff0\infty$, }%
		\label{eqn : poisson/stretching_animal}
		\LimMassA^{(q)}(\beta) &\le g(\beta) \LimMassA^{(q)}(0).
	\end{align}
\end{Lemma}
\begin{Remark}
	\label{rmk : poisson/stretching}
	Our main argument is essentially that for all $\beta > \frac{1}{\sqrt d}$ and $\gamma \in \SetPDF{0}{\ell \beta \base 1}{\ell}$, the contribution of the first coordinate in $\norme{\gamma}$, called $\ell_1^2$ in the proof (see~\eqref{eqn : poisson/continuity/stretching/norm_decomposition}) must be greater than $\frac\ell d$. Consequently, applying a linear transformation of $\R^d$ with determinant $1$ which is a contraction in the direction $\base 1$ and dilation in the other directions, decreases the length of such a path while conserving the density of the point process. This provides a stochastic domination between $\MassPDF{0}{\ell \beta \base 1}{\ell}$ and $\MassPUF{\ell g(\beta)}$.
\end{Remark}
\begin{proof}
	We only prove~\eqref{eqn : poisson/stretching_path}, as~\eqref{eqn : poisson/stretching_animal} is similar. Let $\frac{1}{\sqrt d} \le \beta < 1$ and $\lambda \dpe \p{\frac{(d-1)\beta^2}{1-\beta^2} }^{\frac1{2d} } \ge 1$. Consider the linear map
	\begin{align}
		f : \R^d &\longrightarrow \R^d \nonumber \\ (x^1,\dots,x^d) &\longmapsto \p{\lambda^{-(d-1)}x^1, \lambda x^2, \dots, \lambda x^d }. 	
	\end{align}
	By Theorem~\ref{thm : poisson/mapping}, the process
	\begin{equation}
		\pro'\dpe \sum_{n=1}^\infty\Dirac{\p{ f(z_n), \Mass{z_n} }},
	\end{equation}
	with $(z_n)$ defined by~\eqref{eqn : intro/main/decomposition_pro}, is a Poisson point process with intensity $\Leb \otimes \p{\module{\det f}^{-1}\nu}$. Since $f$ has determinant $1$, $\pro'$ has the same distribution as $\pro$.

	Let $\ell>0$ and $\gamma = (x_0, \dots, x_r)\in \SetPDF{0}{\ell \beta \base 1}{\ell}$ and define, for all $i\in\intint1d$,
	\begin{equation}
		\ell_i^2 \dpe \sum_{j=0}^{r-1}\module{x_j^i - x_{j+1}^i}^2,
	\end{equation}
	with the notation $x_j = (x_j^1, \dots, x_j^d)$. Note that
	\begin{equation}
		\label{eqn : poisson/continuity/stretching/norm_decomposition}
		\sum_{i=1}^d \ell_i^2 = \norme{\gamma}^2 \le \ell^2,
	\end{equation}
	therefore
	\begin{align}
		\norme{f(\gamma)}^2 %
			&= \lambda^{-2(d-1)}\ell_1^2 + \sum_{i=2}^d\lambda^2 \ell_i^2 \nonumber\\
			&\le \lambda^{-2(d-1)} \ell_1^2 + \lambda^2\p{\ell^2 - \ell_1^2}\nonumber\\
			&= \p{\lambda^{-2(d-1)} - \lambda^2}\ell_1^2 + \lambda^2 \ell^2.\nonumber
		\intertext{Since $\lambda \ge 1$ and $\ell_1^2 \ge \beta^2 \ell^2$, }
		\norme{f(\gamma)}^2 %
			&\le \ell^2\cro{\p{\lambda^{-2(d-1)} - \lambda^2}\beta^2 + \lambda^2 }\nonumber\\
			&= \ell^2\cro{\lambda^{-2(d-1)}\beta^2 + \lambda^2(1-\beta^2) }.\nonumber
	\end{align}
	Applying the definition of $\lambda$ yields
	\begin{align}
		\norme{f(\gamma)}^2 %
			&\le \ell^2\cro{ \p{\frac{1-\beta^2}{(d-1)\beta^2}}^{\frac{d-1}{d}}\beta^2 + \p{\frac{1-\beta^2}{(d-1)\beta^2} }^{-1/d}(1-\beta^2)  }\nonumber\\
			&= \ell^2 \p{\frac{1-\beta^2}{d-1}}^{\frac{d-1}{d}} \cdot \cro{\beta^{2/d} + \frac{(d-1)\beta^{2/d}(1-\beta^2) }{1-\beta^2}  }\nonumber\\
			&= \ell^2 d\beta^{2/d}\p{\frac{1-\beta^2}{d-1}}^{\frac{d-1}{d}} = \p{g(\beta) \ell }^2 .\nonumber
	\end{align}
	Consequently,
	\begin{align}
		\label{eqn : poisson/continuity/stretching/UB_length}
		\norme{f(\gamma)} %
			&\le \ell g(\beta).
	\intertext{In particular,}
		\Mass{f(\gamma)}\cro{\pro'} &\le \MassPUF{\ell g(\beta)}\cro{\pro'}.\nonumber
		\intertext{Besides, since $\Mass{f(\gamma)}\cro{\pro'} = \Mass{\gamma}\cro{\pro}$, }
		\Mass{\gamma}\cro{\pro} &\le \MassPUF{\ell g(\beta)}\cro{\pro'}.\nonumber
		\intertext{Taking the supremum in $\gamma$, we obtain}
		\MassPDF{0}{\ell \beta \base 1}{\ell}\cro{\pro} &\le \MassPUF{\ell g(\beta)}\cro{\pro'}.\nonumber
		\intertext{Since $\pro$ and $\pro'$ have the same distribution, dividing by $\ell$ and letting $\ell\to \infty$ gives}
		\LimMassP(\beta ) &\le g(\beta) \LimMassP(0). 
	\end{align}
\end{proof}
\subsection{Proof of Theorem~\ref{thm : intro/main/limit_functions}}
\label{subsec : poisson/limit_functions}
We now prove Theorem~\ref{thm : intro/main/limit_functions}, using Theorems~\ref{thm : intro/known},~\ref{thm : intro/known_pen} and Lemma~\ref{lem : poisson/stretching}.
\begin{proof}[Proof of (i)]
Continuity on $\intervallefo01$ is a consequence of Theorems~\ref{thm : intro/known} and~\ref{thm : intro/known_pen}. Since for all $\beta \in \intervalleff01$ and $q\in\intervalleff0\infty$, %
\[ \LimMassP(\beta) \le \LimMassA^{(q)}(\beta) \le \LimMassA(\beta),\]%
it is sufficient to show that
\begin{equation}
	\label{eqn : poisson/continuity/win_condition}
	\lim_{\beta \to 1} \LimMassA(\beta ) = 0,
\end{equation}
which is a direct consequence of Lemma~\ref{lem : poisson/stretching}.
\end{proof}
\begin{proof}[Proof of (ii)]
	We only prove the result for $\LimMassP$, the other case being analogous. Since $\beta \mapsto \LimMassP(\beta )$ is concave and nonincreasing on $\intervalleff01$, it is sufficient to prove that for all $\frac{1}{\sqrt d}< \beta \le 1$,
\begin{equation}
	\LimMassP(\beta ) < \LimMassP(0).
\end{equation}
Note that the function $g$ defined by~\eqref{eqn : poisson/stretching_def_g} is strictly decreasing on $\intervalleff{\frac{1}{\sqrt d}}{1}$, and $g\p{\frac1{\sqrt d} }=1$, thus Lemma~\ref{lem : poisson/stretching} concludes.
\end{proof}


\subsection{Proof of Theorem~\ref{thm : intro/main/cvg}}
\label{subsec : poisson/cvg}

We only prove the result for $\LimMassP$, the other cases being analogous. Let $0<\delta \le 1$ and
\begin{equation}
	K_\delta \dpe \set{(x,y,\ell) \in \clball{0,2}^2 \times \intervalleof02 }{(x,y)\text{ are colinear and }\norme{x-y} + \delta \le \ell \le 2} \subseteq 2\cX.
\end{equation} A straightforward adaptation of Theorem~\ref{thm : intro/known}, with $\cX$ replaced by $2\cX$, yields
\begin{equation}
	\label{eqn : poisson/cvg/error_term}
	X(L) \dpe \sup\set{ \module{\frac{\MassPDF{Lx}{Ly}{L\ell} }{L} - \ell \LimMassP\p{\frac{\norme{x-y}}{\ell}} } }%
	{ (x,y,\ell) \in K_\delta }%
	\xrightarrow[L \to \infty]{\text{a.s.\ and }\rL^1} 0.
\end{equation}
Let $L>0$. Consider two colinear points $x,y\in \clball{0,1}$ and $\ell>0$. Assume that $\norme{x-y}\le \ell \le 1$. We claim that
\begin{equation}
	\label{eqn : poisson/cvg/UB_general}
	\module{\frac{\MassPDF{Lx}{Ly}{L\ell} }{L} - \ell \LimMassP\p{\frac{\norme{x-y} }{\ell}} }  %
		\le 2 \cro{ \p{ \sqrt\delta   \LimMassP\p{0} }\vee \LimMassP\p{\frac{ 1}{ 1+ \sqrt\delta }} } + X(L).
\end{equation}

\emph{Case 1: }Assume that $(x,y,\ell) \in K_\delta$. Then $\eqref{eqn : poisson/cvg/UB_general}$ follows from the definition of $X(L)$.

\emph{Case 2: }Assume that $(x,y,\ell) \notin K_\delta$. Then by definition of $X(L)$, 
\begin{align}
	\frac{\MassPDF{Lx}{Ly}{L\ell} }{L} %
		&\le \frac{\MassPDF{Lx}{Ly}{L\p{\norme{x-y} + \delta}  } }{L}\eol
		&\le \p{\norme{x-y} + \delta}   \LimMassP\p{\frac{ \norme{x-y} }{ \norme{x-y} + \delta }} + X(L).\nonumber
\end{align}
Besides, since $\beta \mapsto \LimMassP(\beta)$ is nonincreasing on $\intervalleff01$,
\begin{equation*}
	\ell \LimMassP\p{ \frac{ \norme{x-y} }{\ell} } \le  \p{\norme{x-y} + \delta}   \LimMassP\p{\frac{ \norme{x-y} }{ \norme{x-y} + \delta }},
\end{equation*}
thus
\begin{equation}
	\label{eqn : poisson/cvg/UB_module_case2}
	\module{\frac{\MassPDF{Lx}{Ly}{L\ell} }{L} - \ell \LimMassP\p{\frac{ \norme{x-y} }{\ell}} }%
		\le \p{\norme{x-y} + \delta}   \LimMassP\p{\frac{ \norme{x-y} }{ \norme{x-y} + \delta }} + X(L).
\end{equation}
To bound the right-hand side of~\eqref{eqn : poisson/cvg/UB_module_case2}, we further discriminate two subcases, according to the value of $\norme{x-y}$. If$\norme{x-y}< \sqrt\delta$, then~\eqref{eqn : poisson/cvg/UB_module_case2} gives
\begin{equation}
	\module{\frac{\MassPDF{Lx}{Ly}{L\ell} }{L} - \ell \LimMassP\p{\frac{ \norme{x-y} }{\ell}} }%
		\le 2\sqrt\delta   \LimMassP\p{0} + X(L).
\end{equation}
Otherwise,~\eqref{eqn : poisson/cvg/UB_module_case2} and another use of the monotonicity of $\LimMassP$ give
\begin{equation}
	\module{\frac{\MassPDF{Lx}{Ly}{L\ell} }{L} - \ell \LimMassP\p{\frac{ \norme{x-y} }{\ell}} }%
		\le   2\LimMassP\p{\frac{ 1}{ 1+ \sqrt\delta }} + X(L).
\end{equation}
Consequently,~\eqref{eqn : poisson/cvg/UB_general} holds. 

Letting $L \to \infty$ in~\eqref{eqn : poisson/cvg/UB_general} and applying~\eqref{eqn : poisson/cvg/error_term} yields
\begin{equation}
	\label{eqn : poisson/cvg/UB_general_limit}
	\limsup_{L \to \infty} \sup\set{ \module{\frac{\MassGDF{Lx}{Ly}{L\ell} }{L} - \ell \LimMassG\p{\frac{ \norme{x-y} }{\ell}} } }%
					{\begin{array}{c}
						x,y\in \clball{0,1} \text{ colinear}\\ \norme{x-y} \le \ell \le 1	\end{array} }  %
		\le 2 \p{ \sqrt\delta   \LimMassP\p{0} } \vee \LimMassP\p{\frac{ 1}{ 1+ \sqrt\delta }},
\end{equation}
almost surely. Letting $\delta \to 0$ and using the continuity of $\LimMassP$ gives the desired almost sure convergence. Moreover the domination~\eqref{eqn : poisson/cvg/UB_general} gives the convergence in $\rL^1$.\qed


\subsection{Proof of Theorem~\ref{thm : intro/main/q}}
\label{subsec : poisson/q}

We claim that if $\nu$ is infinite, then $q\mapsto \LimMassA^{(q)}(\cdot)$ is constant. Indeed in this case $\pro^*$ is almost surely a dense subset of $\R^d$. Let $\eps, \ell>0$ and $0\le \beta \le 1$. Almost surely, for all animals $\xi \in \SetADFalt{0}{\ell \beta \base 1}{\ell}$, there exists an animal $\xi' \in \SetADFalt{0}{\ell \beta \base 1}{\ell(1 + \eps) }$ such that $\Mass{\xi'} = \Mass{\xi}$, and $\xi'\subseteq \pro^*$. In particular,
\begin{equation*}
	\MassADFpen{0}{\ell \beta \base1}{\ell}{0} \le \MassADFpen{0}{\ell \beta \base1}{\ell(1+\eps)}{\infty} \le \MassADFpen{0}{\ell \beta \base1}{\ell(1+\eps)}{0}
\end{equation*}
Dividing by $\ell$ and letting $\ell \to\infty$ gives, by Theorem~\ref{thm : intro/known_pen},
\begin{equation*}
	\LimMassA^{(0)}(\beta) \le (1+\eps)\LimMassA^{(\infty)}\p{\frac{\beta}{1+\eps}} \le  (1+\eps)\LimMassA^{(0)}\p{\frac{\beta}{1+\eps}}.
\end{equation*}
Since $\LimMassA^{(0)}$ and $\LimMassA^{(\infty)}$ are continuous on $\intervalleff01$, letting $\eps\to0$ gives $\LimMassA^{(0)}(\beta) = \LimMassA^{(\infty)}(\beta)$, thus the claim.

From now on we assume that $\nu$ is finite. In particular $\pro^*$ is a.s.\ a locally finite subset of $\R^d$.

\subsubsection{Proof of Theorem~\ref{thm : intro/main/q}, part 1: Continuity at $q=\infty$}
\label{subsubsec : poisson/cvg_pen/part1}
Let $q\in \intervalleoo0\infty$, $\ell>0$, $\beta \in \intervalleff01$, $\eps>0$. The map
\begin{align*}
	\SetADFalt{0}{\ell \beta \base 1}{\ell} &\longrightarrow \R\\
	\xi &\longmapsto \Mass{\xi} - q\#\p{\xi \cap (\pro^*)^\mathrm{c}}
\end{align*}
only takes values of the form $\sum_{i=1}^n \Mass{x_i}  - qk$, where $x_1,\dots,x_n$ are distinct elements of $\pro^* \cap\clball{0,\ell}$ and $k \in \N$. Consequently, it only takes a finite number of positive values, thus it admits a maximum. Choose in a measurable way an animal $\xi(\ell) \in \SetADFalt{0}{\ell \beta \base 1}{\ell}$ realizing this maximum. Thanks to Lemma~\ref{lem : poisson/cvg_pen/degree}, proven in Section~\ref{subsec : poisson/cvg_pen/technical}, we may assume that $\xi(\ell)$ has bounded degree.
\begin{Lemma}
	\label{lem : poisson/cvg_pen/degree}
	There exists an integer $\Cl{DEGREE}>0$, depending only on the dimension, such that for all finite subsets $A\subseteq \R^d$, there exists an animal with vertex set $A$, minimal length and degree at most $\Cr{DEGREE}$.
\end{Lemma}

We denote by $x_1,\dots,x_n$ the elements of $\xi \cap (\pro^*)^\mathrm{c}$. Since $\Mass\xi \le \MassAUF\ell$,
\begin{equation}
	\label{eqn : poisson/cvg_pen/UB_bad_vertices}
	n \le \frac{\MassAUF\ell}{q}.
\end{equation}
Consider a Poisson process $\pro_\eps$ with intensity $\eps \Leb\otimes \nu$, independent from $\pro$. By the superposition principle (see e.g.\ Theorem~3.3 in \cite{LastPenrose}), $\pro + \pro_\eps$ is a Poisson process with intensity $\Leb\otimes (1+\eps)\nu$. Define
\begin{equation}
	\forall i\in\intint1n, \quad R_i \dpe \min_{y \in \pro_\eps^*} \norme{x_i - y}, \text{ and } R\dpe \sum_{i=1}^nR_i.
\end{equation}
For all $i\in \intint1n$,
\begin{equation*}
	\E{R_i} = \int_0^\infty\Pb{R_i > s}\d s = \int_0^\infty \exp\p{-\eps \nu\p{\intervalleoo0\infty} \Leb\p{\clball{0,1} } s^d}\d s \eqqcolon I(\eps) < \infty,
\end{equation*}
thus
\begin{equation*}
	\Econd{R}{\pro} = nI(\eps) \le \frac{\MassAUF\ell I(\eps)}{q}.
\end{equation*}
In particular, by~\eqref{eqn : intro/known/General_bound},
\begin{equation*}
	\E{R} \le \frac{\GeneralUB \ell I(\eps)}{q}.
\end{equation*}
Markov's inequality yields
\begin{equation}
	\label{eqn : poisson/cvg_pen/good_event}
	\Pb{R \le \frac{2\GeneralUB \ell I(\eps)}{q}} \ge \frac12.
\end{equation}

Consider the animal $\xi'(\ell)$ built from $\xi(\ell)$ by replacing each $x_i$ by a point $y_i \in \pro_\eps^*$ such that $\norme{x_i-y_i}=R_i$. By the triangle inequality and the bound on the degree of $\xi(\ell)$, we have
\begin{align}
	\norme{\xi'(\ell)}&\le \norme{\xi(\ell)} + \Cr{DEGREE}R.\nonumber%
	\intertext{In particular, on the event $\acc{R \le \frac{2\GeneralUB \ell I(\eps)}{q}}$,  }
	\label{eqn : poisson/cvg_pen/UB_length}
	\norme{\xi'(\ell)} &\le \ell\p{ 1 + \frac{2 \Cr{DEGREE} \GeneralUB I(\eps)}{q} }.
\end{align}
Inequalities~\eqref{eqn : poisson/cvg_pen/good_event} and~\eqref{eqn : poisson/cvg_pen/UB_length} imply
\begin{equation*}
	\Pb{ \MassADFpen{0}{\ell \beta \base 1}{\ell}{q}\cro\pro \le  \MassADFpen{0}{\ell \beta \base 1}{\p{ 1 + \frac{2 \Cr{DEGREE} \GeneralUB I(\eps)}{q} }\ell}{\infty}\cro{\pro + \pro_\eps} } \ge \frac12.
\end{equation*}
Applying~\eqref{eqn : intro/known/cvg}, we obtain
\begin{align}
	\LimMassA^{(q)}(\beta)\cro\nu%
		&\le \p{ 1 + \frac{2 \Cr{DEGREE} \GeneralUB I(\eps)}{q} }\LimMassA^{(\infty)}\p{\frac{\beta}{ 1 + \frac{2 \Cr{DEGREE} \GeneralUB I(\eps)}{q} }} \cro{(1+\eps)\nu}.\nonumber
	\intertext{By Lemma~\ref{lem : poisson/scaling}, }
	\LimMassA^{(q)}(\beta)\cro\nu%
		&\le (1+\eps)^{1/d}\p{ 1 + \frac{2 \Cr{DEGREE} \GeneralUB I(\eps)}{q} }\LimMassA^{(\infty)}\p{\frac{\beta}{ 1 + \frac{2 \Cr{DEGREE} \GeneralUB I(\eps)}{q} }} \cro{\nu}.\end{align}
	From now on we drop the notation $\cro\nu$. Since $\LimMassA^{(\infty)}$ is continuous, letting $q\to \infty$ then $\eps \to 0$ gives
	\begin{align}
	\declim{q\to\infty}\LimMassA^{(q)}(\beta)%
		&\le \LimMassA^{(\infty)}(\beta),\nonumber
	\intertext{thus}
	\declim{q\to\infty}\LimMassA^{(q)}(\beta)%
		&= \LimMassA^{(\infty)}(\beta).
\end{align}
The functions $\LimMassA^{(q)}$ and $\LimMassA^{(\infty)}$ are continuous on the compact set $\intervalleff01$ and the convergence is monotone, therefore by Dini's theorem, \begin{equation}
	\lim_{q \to \infty} \LimMassA^{(q)} = \LimMassA^{(\infty)}
\end{equation}
for the uniform convergence on $\intervalleff01$. \qed

\subsubsection{Proof of Theorem~\ref{thm : intro/main/q}, part 2: Continuity on $\intervallefo0\infty$}
With the notations of~\eqref{eqn : intro/main/decomposition_pro}, consider the process
\begin{equation}
	\hat\pro \dpe \sum_{n=1}^\infty \Dirac{(z_n, 1)}.
\end{equation}
Then by Theorem~\ref{thm : poisson/mapping}, $\hat\pro$ is a Poisson point process on $\R^d\times \intervallefo0\infty$ with intensity $\Leb \otimes \p{\nu\p{\intervalleoo0\infty} \Dirac{1} }$. In particular, by~\eqref{eqn : intro/known/General_bound},
\begin{equation}
	\LimMassA(0)\cro{ \nu\p{\intervalleoo0\infty}\Dirac1 } \le \Cr{Dirac_bound} \nu\p{\intervalleoo0\infty}^{1/d} < \infty.
\end{equation}
To prove that $q\mapsto \LimMassA^{(q)}$ is continuous on $\intervallefo0\infty$, we will show that it is $\LimMassA(0)\cro{ \nu\p{\intervalleoo0\infty}\Dirac1 }$-Lipschitz. Lemma~\ref{lem : poisson/cvg_pen/UB_bad_vertices}, proven in Section~\ref{subsec : poisson/cvg_pen/technical}, gives an upper bound for the number of useful vertices of an animal outside $\projpro$.
\begin{Lemma}
	\label{lem : poisson/cvg_pen/UB_bad_vertices}
	Almost surely, for all $u\in\clball{0,1}$, $\ell>0$ and $q\in \intervallefo0\infty$, there exists an animal $\xi \in \SetADFalt{0}{\ell u}{\ell}$ that maximizes $\Mass\xi - q\# \p{\xi \cap (\projpro)^\mathrm{c}}$ and satisfies
	\begin{equation}
		\label{eqn : poisson/cvg_pen/UB_bad_vertices}
		\# \p{\xi \cap (\projpro)^\mathrm{c}} \le \MassAUF\ell\cro{ \hat\pro }.
	\end{equation}
\end{Lemma}
Let $0\le q_1 \le q_2 < \infty$. The inequality 
\begin{equation}
	\label{eqn : poisson/cvg_pen/easy_bound} 
	\LimMassA^{(q_1)}\ge \LimMassA^{(q_2)}
\end{equation} 
is straightforward. Let $\beta \in \intervalleff01$ and $\ell>0$. Consider an animal $\xi$ as in Lemma~\ref{lem : poisson/cvg_pen/UB_bad_vertices}, with $q=q_1$. Then
\begin{align}
	\Mass\xi - q_2\# \p{\xi \cap (\projpro)^\mathrm{c}} %
		&= \MassADFpen{0}{\ell \beta \base 1}{\ell}{q_1} - (q_2 - q_1)\# \p{\xi \cap (\projpro)^\mathrm{c}}\nonumber\\
		&\ge \MassADFpen{0}{\ell \beta \base 1}{\ell}{q_1} - (q_2 - q_1)\MassAUF\ell\cro{ \hat\pro }.\nonumber
	\intertext{In particular,}
	\MassADFpen{0}{\ell \beta \base 1}{\ell}{q_2} %
		&\ge \MassADFpen{0}{\ell \beta \base 1}{\ell}{q_1} - (q_2 - q_1)\MassAUF\ell\cro{ \hat\pro }.
\end{align}
By~\eqref{eqn : intro/known/cvg}, dividing by $\ell$ and letting $\ell \to \infty$ gives
\begin{equation}
	\label{eqn : poisson/cvg_pen/other_bound} 
	\LimMassA^{(q_2)}(\beta \base 1)\ge \LimMassA^{(q_1)}(\beta \base 1) - (q_2-q_1)\LimMassA(0)\cro{ \nu\p{\intervalleoo0\infty}\Dirac1 }.
\end{equation}
Consequently, $q\mapsto \LimMassA^{(q)}$ is $\LimMassA(0)\cro{ \nu\p{\intervalleoo0\infty}\Dirac1 }$-Lipschitz thus continuous on $\intervallefo0\infty$. \qed


\subsection{Proof of the technical lemmas}
\label{subsec : poisson/cvg_pen/technical}
\begin{proof}[Proof of Lemma~\ref{lem : poisson/cvg_pen/degree}]
	Since the unit sphere $\S$ of $\R^d$ is compact, it may be covered by a finite number $\Cr{DEGREE}$ of open balls with radius $\frac12$. By pigeonhole principle, for all $y_1, \dots, y_{\Cr{DEGREE}+1} \in \S$, there exists distinct $i,j\in \intint{1}{\Cr{DEGREE} +1 }$ such that
	\begin{equation}
		\label{eqn : poisson/cvg_pen/degree/compactness_sphere}
		\norme{y_i - y_j} < 1.
	\end{equation}
	Let $A$ be a finite subset of $\R^d$ and $\xi$ an animal with vertex set $A$. It is sufficient to show that if $\xi$ has a vertex with degree at least $\Cr{DEGREE}+1$, then there exists another animal $\xi'$ with vertex set $A$ and smaller length. Assume that such a vertex $x$ exists. By~\eqref{eqn : poisson/cvg_pen/degree/compactness_sphere} there exist distinct neighbours $x_1 ,x_2$ of $x$ in $\xi$, such that
	\begin{equation}
		\label{eqn : poisson/cvg_pen/degree/x1_x2}
		\norme{\frac{x_1 - x}{\norme{x_1 - x} } - \frac{x_2 - x}{\norme{x_2 - x} } } < 1.
	\end{equation}
	Up to exchanging $x_1$ and $x_2$, we may assume $\norme{x_1 - x} \le \norme{x_2 -x}$. We have by triangle inequality
	\begin{align*}
		\norme{x_1 - x_2}%
			&\le \norme{x_1 - \p{x + \frac{\norme{x_1 - x}(x_2 - x)}{\norme{x_2 - x}} } } + \norme{ \p{x + \frac{\norme{x_1 - x}(x_2-x)}{\norme{x_2 - x}} } - x_2}\\
			&\le \norme{x_1 - x} \cdot \norme{\frac{x_1 - x}{\norme{x_1 - x} } - \frac{x_2 - x}{\norme{x_2 - x} } } + \norme{x_2 - x} \cdot \module{1- \frac{ \norme{ x_1-x} }{ \norme{ x_2-x} }}.
		\intertext{Applying~\eqref{eqn : poisson/cvg_pen/degree/x1_x2} and $\norme{x_1 - x} \le \norme{x_2 -x}$ yields }
		\norme{x_1 - x_2}%
			&< \norme{x_1 - x} + \norme{x_2 - x} \p{1- \frac{ \norme{ x_1-x} }{ \norme{ x_2-x} }} = \norme{x_2 - x}.
	\end{align*}
	In particular, removing the edge $\acc{x,x_2}$ and adding the edge $\acc{x_1,x_2}$ in $\xi$ strictly decreases the length while preserving connectedness and vertex set. This concludes the proof.
\end{proof}
\begin{proof}[Proof of Lemma~\ref{lem : poisson/cvg_pen/UB_bad_vertices}]
	Let $u\in\clball{0,1}$, $\ell>0$ and $q\in \intervallefo0\infty$. Consider an animal $\xi = (V,E) \in \SetADFalt{0}{\ell u}{\ell}$ that maximizes $\Mass\xi - q\# \p{\xi \cap (\projpro)^\mathrm{c}}$, the existence of which has been shown at the beginning of Section~\ref{subsubsec : poisson/cvg_pen/part1}. Up to modifying the edge set $E$, we can assume that $\xi$ is a tree and $V$ has minimal cardinal. It suffices to show that $\xi$ satisfies~\eqref{eqn : poisson/cvg_pen/UB_bad_vertices}. Consider
	\begin{equation*}
		x\in \argmin_{z \in V} \norme{z} \text{ and } y \in \argmin_{z\in V} \norme{\ell u-z}.
	\end{equation*}
	We say that a vertex $z\in V$ is \emph{good} if $z\in\acc{x,y}\cup \projpro$, and \emph{bad} otherwise. Let $\Good$ and $\Bad$ denote the set of good vertices and bad vertices respectively. Note that removing any bad vertex with degree $1$ creates an animal $\xi'\in \SetADFalt{0}{\ell u}{\ell}$ with greater penalized mass. So does removing any bad vertex with degree $2$ and connecting its two neighbours together. Consequently, by definition of $\xi$, all bad vertices have degree at least $3$, thus
	\begin{align}
		\sum_{z \in V}\deg_\xi(z) %
			&= \sum_{z \in \Good}\deg_\xi(z) + \sum_{z \in \Bad}\deg_\xi(z)\nonumber \\
			&\ge \# \Good + 3\# \Bad. \label{eqn : poisson/cvg_pen/UB_bad_vertices/LB_deg}
	\end{align}
	Besides, since $\xi$ is a tree,
	\begin{equation}
		\label{eqn : poisson/cvg_pen/UB_bad_vertices/deg_nb_edges}
		\sum_{z \in V}\deg_\xi(z) = 2\# E = 2(\# V -1) = 2(\# \Good + \# \Bad -1).
	\end{equation}
	Combining~\eqref{eqn : poisson/cvg_pen/UB_bad_vertices/LB_deg} and~\eqref{eqn : poisson/cvg_pen/UB_bad_vertices/deg_nb_edges}, we get
	\begin{equation*}
		\#\Bad \le \# \Good -2.
	\end{equation*}
	Furthermore, $\# \Good \le 2 + \Mass{\xi}[\hat \pro]$, thus \eqref{eqn : poisson/cvg_pen/UB_bad_vertices}.
\end{proof}

\end{document}